\documentclass[final,nomarks]{dmtcs-episciences}
\usepackage[utf8]{inputenc}
\usepackage{subfigure}
\usepackage{amsmath,amsthm,amssymb}
\usepackage{tikz}
\usetikzlibrary{arrows, shapes, shapes.geometric, automata, positioning}
\tikzset{elliptic state/.style={draw,ellipse}}

\theoremstyle{plain}
\newtheorem{thm}{Theorem}[section]
\newtheorem{prop}[thm]{Proposition}
\newtheorem{lem}[thm]{Lemma}
\newtheorem{cor}[thm]{Corollary}
\newtheorem{con}[thm]{Conjecture}
\theoremstyle{definition}
\newtheorem{rem}[thm]{Remark}

\author{\L{}ukasz Merta}
\title[Formal inverses of certain automatic sequences]{Formal inverses of the generalized Thue--Morse sequences and variations of the Rudin--Shapiro sequence}

\affiliation{
Jagiellonian University in Krak\'{o}w, Poland}
\received{2018-11-5}
\revised{2019-11-22}
\accepted{2020-4-8}
\begin{document}
\publicationdetails{22}{2020}{1}{15}{4954}
\maketitle
\begin{abstract}
A formal inverse of a given automatic sequence (the sequence of coefficients of the composition inverse of its associated formal power series) is also automatic. The comparison of properties of the original sequence and its formal inverse is an interesting problem. Such an analysis has been done before for the Thue--Morse sequence. In this paper, we describe arithmetic properties of formal inverses of the generalized Thue--Morse sequences and formal inverses of two modifications of the Rudin--Shapiro sequence. In each case, we give the recurrence relations and the automaton, then we analyze the lengths of factors consisting of identical letters as well as the frequencies of letters. We also compare the obtained results with the original sequences.
\end{abstract}

\section{Introduction}

Let \(k \geq 2\). An infinite sequence 
\begin{math} (a_n)_{n \in \naturals} \end{math}
 is called \(k\)-automatic, if the \(n\)-th term is generated by a finite automaton with the base-\(k\) expansion of \(n\) as input.

The Thue--Morse sequence 
\begin{math} (t_n)_{n \in \naturals} \end{math}
(sequence A010060 in \cite{OE}) is one of the best-known examples of a non-trivial \(2\)-automatic sequence. Its \(n\)-th term is equal to the sum of digits in the binary expansion of \(n\) taken modulo \(2\). This sequence is generated by an automaton with \(2\) states and it satisfies the following recurrence relations:
\begin{displaymath} 
t_{2n} = t_n, \quad t_{2n+1} = 1 - t_n, \quad n \in \naturals.
\end{displaymath}

The Thue--Morse sequence can be easily generalized. For a prime number \(p\), we can define the sequence
\begin{math} (t_n^{(p)})_{n \in \naturals} \end{math}
as a sequence such that \(t_n^{(p)}\) is equal to the sum of digits in the base-\(p\) expansion of \(n\) taken modulo \(p\). This sequence is then \(p\)-automatic and in the case \(p=2\) we retrieve the original sequence
\begin{math} (t_n)_{n \in \naturals} \end{math}.

Another known example of a \(2\)-automatic sequence is the Rudin--Shapiro sequence \begin{math} (r_n)_{n \in \naturals} \end{math} 
(sequence A020985 in \cite{OE}). Its \(n\)-th term is equal to \(1\) if the number of (possibly overlapping) occurrences of \(11\) in the binary expansion of \(n\) is even and \(-1\) otherwise. It is generated by an automaton with \(4\) states and it satisfies the following relations:
\begin{displaymath} 
r_0 = 1, \quad r_{2n} = r_{4n+1} = r_n, \quad r_{4n+3} = -r_{2n+1}, \quad n \in \naturals.
\end{displaymath}

For an infinite sequence 
\begin{math} (a_n)_{n \in \naturals} \end{math}
with values in \(\mathbb{F}_p\), we can consider a formal power series 
\begin{math} F = \sum_{n=0}^{\infty}a_nX^n \in \mathbb{F}_p[\![X]\!] \end{math}. 
If \(a_0 = 0\) and \(a_1 \neq 0\), then there exists a unique formal power series 
\begin{math} G = \sum_{n=0}^{\infty}b_nX^n \in \mathbb{F}_p[\![X]\!] \end{math}
such that 
\begin{math} F(G(X)) = G(F(X)) = X \end{math}.
The obtained sequence 
\begin{math} (b_n)_{n \in \naturals} \end{math}
is then called the formal inverse of the sequence
\begin{math} (a_n)_{n \in \naturals} \end{math}.
If
\begin{math} (a_n)_{n \in \naturals} \end{math}
is \(p\)-automatic, then its formal inverse is also \(p\)-automatic \cite[Theorem 12.2.5]{AS}.

The main problem considered in this paper is to study the formal inverses of some automatic sequences and compare them to the original sequences. This problem was already discussed in \cite{PTM} by M. Gawron and M. Ulas, who studied the formal inverse of the Thue--Morse sequence
\begin{math} (t_n)_{n \in \naturals} \end{math},
which they denoted by 
\begin{math} (c_n)_{n \in \naturals} \end{math}.
The authors found many interesting properties of this sequence, including the recurrence relations, the algebraic relation for the associated formal power series and the automaton generating the sequence
\begin{math} (c_n)_{n \in \naturals} \end{math}.
They also studied the possible lengths of strings of consecutive \(0\)'s and consecutive \(1\)'s in this sequence. Analogous results have been obtained for two variations of the Baum--Sweet sequence in \cite{BS}. A recent note by N. Rampersad and M. Stipulanti \cite{PD} discusses similar results for the inverse of the period-doubling sequence
\begin{math} (d_n)_{n \in \naturals} \end{math},
whose \(n\)-th term is equal to the exponent of the highest power of \(2\) dividing \(n+1\). In this paper, we consider sequences mentioned earlier --- the generalized Thue--Morse sequences and the variations of the Rudin-Shapiro sequence.

The paper is divided into two sections. In the first section, we study properties of the formal inverses of the sequences
\begin{math} (t_n^{(p)})_{n \in \naturals} \end{math},
which we denote by
\begin{math} (c_n^{(p)})_{n \in \naturals} \end{math}.
We start with finding the recurrence relation for the sequence
\begin{math} (c_n^{(p)})_{n \in \naturals} \end{math}
in the general case. Then, we study the sequence
\begin{math} (c_n^{(p)})_{n \in \naturals} \end{math}
in the cases \(p = 3\) and \(p = 5\). In each case, we start with the automaton that generates the sequence. We then study properties of the sequence, based on the obtained automaton as well as properties of the automaton itself. We discuss the maximal number of consecutive \(0\)'s and the frequency of \(0\)'s in both sequences. We also introduce some results involving consecutive nonzero terms. At the end of this section,  we provide a list of conjectures for the general case, based on the obtained results.

In the second section, we consider the Rudin--Shapiro sequence, with all \(-1\) terms changed to \(0\) so it can be considered as a sequence with values in \(\mathbb{F}_2\). This sequence does not meet the conditions for a formal inverse to exist, hence we consider the following variations:
\begin{displaymath}
r'_n = \left\{ \begin{array}{ll}
0 & \text{ if } n = 0, \\ r_n & \text{ otherwise},
\end{array} \right.
\quad
r''_n = \left\{ \begin{array}{ll}
0 & \text{ if } n = 0, \\ r_{n-1} & \text{ otherwise}.
\end{array} \right.
\end{displaymath}
We then study the formal inverses of those sequences, which we denote by
\begin{math} (u_n)_{n \in \naturals} \end{math}
and
\begin{math} (v_n)_{n \in \naturals} \end{math}, 
respectively. We focus on the same aspects as in the previous section --- we discuss properties involving the frequency of letters and the maximal number of consecutive \(0\)'s and \(1\)'s in both sequences. We also verify whether the automata generating these sequences are synchronizing.

Many results in this paper are found and proved using some specific computer programs. Each automaton is computed using the Mathematica package \texttt{IntegerSequences} written by Eric Rowland (\verb|https://people.hofstra.edu/Eric_Rowland/packages.html|). This package allows us to compute the automaton generating a given sequence by using the algebraic relation for the associated formal power series (the method is implemented from the proof of\cite[Theorem 12.2.5]{AS}). We also use the automatic theorem-proving software called \texttt{Walnut}, written by Hamoon Mousavi \cite{WAL}, to prove some technical lemmas involving the obtained sequences. Furthermore, we use a couple of small computer programs written in \texttt{C++} or \texttt{R} to test some properties of the obtained automata (such as finding the formulas for the sequences
\begin{math} (g_n)_{n \in \naturals} \end{math}
and
\begin{math} (h_n)_{n \in \naturals} \end{math}
in the proof of Theorem \ref{c_consec_12} or finding the structure of the automaton \(A_5\) in Subsection \ref{section_thue5}). The source code for all the computations is available from the author.

Each automaton in this paper is considered with input processed starting from the least significant digit and every state is labeled with the corresponding subsequence. Moreover, we describe every automaton as a \(6\)-tuple 
\begin{math} (Q, \Sigma, \delta, q_0, \Delta, \tau) \end{math}, using the same notation as in \cite{AS}. Since we focus only on one particular automaton at a time, we use the same notation for all automata.

Similarly to \cite{BS}, this paper is an extended version of two chapters of the Master's thesis defended at the Jagiellonian University in 2017 \cite{MT}.

\section{Generalized Thue--Morse sequences} \label{section_thue}
\subsection{Basic results}

Let \(p\) be a prime number. We let \(s_p(n)\) denote the sum of digits of \(n \in \naturals\) in base \(p\). We define the sequence
\begin{math} (t_n^{(p)})_{n \in \naturals} \end{math}
as
\begin{displaymath}
t_n^{(p)} = s_p(n) \quad \text{in } \mathbb{F}_p.
\end{displaymath}
It is clear that the sequence 
\begin{math} (t_n^{(p)})_{n \in \naturals} \end{math}
satisfies the following relations:
\begin{equation} \label{cp_recur}
t_0^{(p)} = 0, \quad t_{pn+i}^{(p)} = t_n^{(p)} + i.
\end{equation}

Let
\begin{math} F_p = \sum_{n=0}^{\infty}t_n^{(p)}X^n \in \mathbb{F}_p[\![X]\!] \end{math}. 
From the recurrence relations, it is easy to prove (see \cite[Example 12.1.3]{AS}) that \(F_p\) satisfies the following algebraic relation: 
\begin{equation} \label{Fp_equation}
(1 - X)^{p+1}F_p^p - (1 - X)^2F_p + X = 0.
\end{equation}
Since we have \(t_0^{(p)} = 0\) and \(t_1^{(p)} = 1\), we know that for any prime \(p\) there exists a formal series \(C_p\) such that 
\begin{math} C_p(F_p(X)) = F_p(C_p(X)) = X \end{math}. 
We can easily find the analogous relation for \(C_p\), by left composing equation \eqref{Fp_equation} with \(C_p\). We thus have
\begin{equation} \label{Cp_equation}
(1 - C_p)^{p+1}X^p - (1 - C_p)^2X + C_p = 0.
\end{equation}

Let 
\begin{math} S_p = X(1 - C_p) \end{math} 
and let 
\begin{math} (s_n^{(p)})_{n \in \naturals} \end{math}
be the sequence of coefficients of \(S_p\). From equation \eqref{Cp_equation}, the series \(S_p\) satisfies the following algebraic relation: 
\begin{equation} \label{Sp_equation}
S_p^{p+1} - S_p^2 - S_p + X = 0.
\end{equation}

The following proposition gives the recurrence relations for the sequence 
\begin{math} (s_n^{(p)})_{n \in \naturals} \end{math}.

\begin{prop} \label{s_recurrence}
The sequence
\begin{math} (s_n^{(p)})_{n \in \naturals} \end{math}
satisfies \(s_0^{(p)} = 0\), \(s_1^{(p)} = 1\) and
\begin{displaymath}
s_n^{(p)} = s_1^{(p)}w_{n-1}^{(p)} + s_2^{(p)}w_{n-2}^{(p)} + \dots + s_{n-1}^{(p)}w_1^{(p)} \quad \text{for } n \geq 2,
\end{displaymath}
where the sequence
\begin{math} (w_n^{(p)})_{n \in \naturals} \end{math}
satisfies \(w_0^{(p)} = -1\) and
\begin{displaymath}
w_n^{(p)} = \left\{ \begin{array}{ll}
-s_n^{(p)} & \text{if } \, p \nmid n, \\
-s_n^{(p)} + s_{n/p}^{(p)} & \text{if } \, p \mid n \end{array} \right. \quad \text{for } n \geq 1.
\end{displaymath}
\end{prop}

\begin{proof}
From equation \eqref{Sp_equation}, we have
\begin{equation} \label{Sp_equation2}
S_p(X)(S_p(X^p) - S_p(X) - 1) = -X.
\end{equation}
Let
\begin{math} W_p(X) = S_p(X^p) - S_p(X) - 1 \end{math}
and let 
\begin{math} (w_n^{(p)})_{n \in \naturals} \end{math}
be the sequence of coefficients of \(W_p\). We thus have
\begin{displaymath}
\sum_{n=0}^{\infty}w_n^{(p)}X^n = \sum_{n=0}^{\infty}s_n^{(p)}X^{pn} - \sum_{n=0}^{\infty}s_n^{(p)}X^n - 1.
\end{displaymath}
Comparing the coefficients on both sides of this equality, we obtain the recurrence relations for 
\begin{math} (w_n^{(p)})_{n \in \naturals} \end{math}
presented in the statement of our proposition. From \eqref{Sp_equation2}, we obtain the following equation:
\begin{displaymath}
\left( \sum_{n=0}^{\infty}s_n^{(p)}X^n \right)
\left( \sum_{n=0}^{\infty}w_n^{(p)}X^n \right) = \sum_{n=0}^{\infty} \left( \sum_{k=0}^n s_k^{(p)}w_{n-k}^{(p)} \right)X^n = -X.
\end{displaymath}
Hence, from the last equality, we can see that \(s_0^{(p)} = 0\), \(s_1^{(p)} = 1\) and
\begin{displaymath}
s_0^{(p)}w_n^{(p)} + s_1^{(p)}w_{n-1}^{(p)} + \dots + s_n^{(p)}w_0^{(p)} = 0
\end{displaymath}
for \(n \geq 2\). Since \(s_0^{(p)} = 0\) and \(w_0^{(p)} = -1\), this equation is equivalent to the recurrence relation in our proposition.
\end{proof}

Let
\begin{math} (c_n^{(p)})_{n \in \naturals} \end{math}
be the sequence of coefficients of \(C_p\). Since we know the relation 
\begin{math} S_p = X(1 - C_p) \end{math},
we can easily compute the terms of the sequence
\begin{math} (c_n^{(p)})_{n \in \naturals} \end{math}.
We have \(c_0^{(p)} = 0\), \(c_1^{(p)} = 1\) and
\begin{math} c_n^{(p)} = -s_{n+1}^{(p)}\end{math}
for \(n \geq 2\).

In the remaining part of this section, we are going to focus on properties of the sequence
\begin{math} (c_n^{(p)})_{n \in \naturals} \end{math}
for \(p \in \{2,3,5\}\). The case \(p = 2\) was already described in \cite{PTM}. The authors proved that the sequence 
\begin{math} (c_n^{(2)})_{n \in \naturals} \end{math}
has the following properties:
\begin{itemize}
\item The sequence is generated by an automaton with \(5\) states (with input represented in base \(4\)).
\item There are arbitrarily long sequences of consecutive \(0\)'s.
\item The maximal number of consecutive \(1\)'s is equal to \(4\).
\item The frequency of \(0\)'s is equal to \(1\).
\end{itemize}

The last property was not stated directly in \cite{PTM}, but it can be proved using the fact that the automaton generating the sequence
\begin{math} (c_n^{(2)})_{n \in \naturals} \end{math}
is synchronizing (the proof of this fact is similar to the proof of Corollary \ref{c_zero_freq} in the next subsection).

\vspace{0.3 cm}
Below we show another automaton generating the sequence 
\begin{math} (c_n^{(2)})_{n \in \naturals} \end{math},
(computed in Mathematica). This time the automaton has \(8\) states and input is represented in base \(2\). This automaton is synchronizing as well and the shortest synchronizing word is \(011\).

\begin{figure}[ht!]
\begin{center}
\begin{tikzpicture}[->, shorten >= 1pt, node distance=2.5 cm, on grid, auto]
	\node[elliptic state, inner sep=3pt, line width = 1.5pt] (c_1) {\(c^{(2)}_n\)};
	\node[elliptic state, inner sep=2pt] (c_2) [above right=of c_1] {\(c^{(2)}_{2n}\)};
	\node[elliptic state, inner sep=2pt] (c_3) [right=of c_2] {\(c^{(2)}_{4n}\)};
    \node[elliptic state, inner sep=2pt] (c_4) [right=4cm of c_3] {\(c^{(2)}_{8n+3}\)};	
    \node[elliptic state, inner sep=2pt] (c_5) [below right=of c_1] {\(c^{(2)}_{2n+1}\)};
	\node[elliptic state, inner sep=2pt] (c_6) [right=of c_5] {\(c^{(2)}_{4n+3}\)};
	\node[elliptic state, inner sep=2pt] (c_7) [below right=of c_2] {\(c^{(2)}_{4n+1}\)};
    \node[elliptic state, inner sep=2pt] (c_8) [right=3cm of c_7] {\(c^{(2)}_{8n+1}\)};	
	 \path[->]
	 (c_1) edge node {0} (c_2)
	       edge node [swap] {1} (c_5)
	 (c_2) edge node [swap] {0} (c_3)
	       edge node [swap] {1} (c_7)
	 (c_3) edge [bend right] node [swap] {0} (c_2)
	 	   edge node {1} (c_4)
	 (c_4) edge [loop right] node {0,1} (c_4)
	 (c_5) edge node {0} (c_7)
	       edge node {1} (c_6)
	 (c_6) edge [bend right=70] node [swap] {0} (c_4)
	       edge [bend left] node {1} (c_5)
	 (c_7) edge node {0} (c_8)
	 	   edge [bend left=10] node {1} (c_4)
     (c_8) edge [bend left] node {0,1} (c_7);     
\end{tikzpicture}
\vspace{-0.2 cm}
\caption{The automaton \(A_2\) generating the sequence \((c^{(2)}_n)_{n \in \naturals}\). \label{c2_auto}}
\end{center}
\end{figure}

In the following subsections, we are going to prove analogous properties for the sequences 
\begin{math} (c_n^{(3)})_{n \in \naturals} \end{math} 
and 
\begin{math} (c_n^{(5)})_{n \in \naturals} \end{math}.

\subsection{\texorpdfstring{The case \(p = 3\)}{The case p=3}}
\label{section_thue3}

We start with properties of the sequence 
\begin{math} (c_n^{(p)})_{n \in \naturals} \end{math}
in the case \(p=3\). This is the sequence A053838 in \cite{OE}. From now on, to simplify the notation, we write 
\begin{math} (c_n)_{n \in \naturals} \end{math}
instead of 
\begin{math} (c_n^{(3)})_{n \in \naturals} \end{math}.

The formal power series 
\begin{math} C_3 = \sum_{n=0}^{\infty}c_nX^n \end{math}
satisfies the following algebraic relation:
\begin{displaymath}
(1 - C_3)^4X^3 - (1 - C_3)^2X + X = 0.
\end{displaymath}
This equation can be used to determine the automaton
\begin{math} A_3 = (Q, \Sigma_3, \delta, [c_n], \Sigma_3, \tau) \end{math}
that generates the sequence 
\begin{math} (c_n)_{n \in \naturals} \end{math}. 
We use Wolfram Mathematica to perform all the computations. As a result, we obtain an automaton with \(28\) states -- it is shown below in Figure \ref{c_auto}. Since we labeled the states with the corresponding subsequences (state \(q\) is labeled with the subsequence of the form 
\begin{math} (c_{3^kn + l})_{n \in \naturals} \end{math}, 
where \(k \in \naturals\), \(0 \leq l < 3^k\), such that
\begin{math} c_{3^kn + l} = \tau(\delta(q, (n)_3)) \end{math}
for all \(n\)) instead of the output values, we also provide a table with the values of the function \(\tau\).

\begin{figure}[ht!] 
\includegraphics[width=\textwidth]{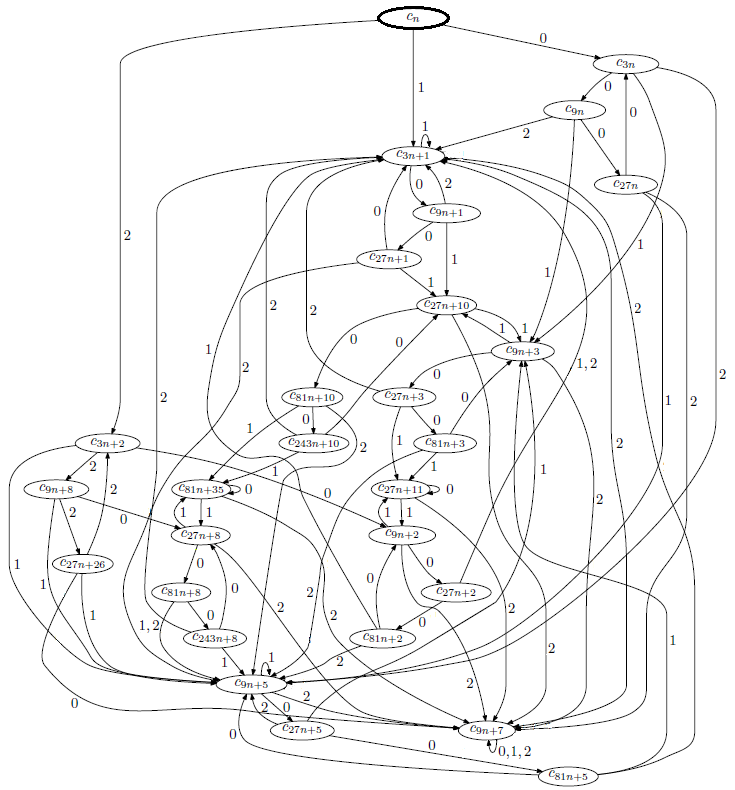}
\caption{The automaton \(A_3\) generating the sequence \((c_n)_{n \in \naturals}\). \label{c_auto}}
\end{figure}

\begin{table}[ht!] 
\centering
\begin{tabular}{|c|c||c|c||c|c||c|c|}
\hline
\(q\) & \(\tau(q)\) & \(q\) & \(\tau(q)\) & \(q\) & \(\tau(q)\) & \(q\) & \(\tau(q)\) \\
\hline
\([c_n]\) & \(0\) & \([c_{9n+3}]\) & \(1\) & \([c_{27n+3}]\) & \(1\) & \([c_{81n+3}]\) & \(1\) \\
\hline
\([c_{3n}]\) & \(0\) & \([c_{9n+5}]\) & \(2\) & \([c_{27n+5}]\) & \(2\) & \([c_{81n+5}]\) & \(2\) \\
\hline
\([c_{3n+1}]\) & \(1\) & \([c_{9n+7}]\) & \(0\) & \([c_{27n+8}]\) & \(2\) & \([c_{81n+8}]\) & \(2\) \\
\hline
\([c_{3n+2}]\) & \(1\) & \([c_{9n+8}]\) & \(2\) & \([c_{27n+10}]\) & \(2\) & \([c_{81n+10}]\) & \(2\) \\
\hline
\([c_{9n}]\) & \(0\) & \([c_{27n}]\) & \(0\) & \([c_{27n+11}]\) & \(0\) & \([c_{81n+35}]\) & \(0\) \\
\hline
\([c_{9n+1}]\) & \(1\) & \([c_{27n+1}]\) & \(1\) & \([c_{27n+26}]\) & \(0\) & \([c_{243n+8}]\) & \(2\) \\
\hline
\([c_{9n+2}]\) & \(1\) & \([c_{27n+2}]\) & \(1\) & \([c_{81n+2}]\) & \(1\) & \([c_{243n+10}]\) & \(2\) \\
\hline
\end{tabular}
\caption{The output values of the automaton \(A_3\). \label{c_out}}
\end{table}

We can now use the automaton \(A_3\) to determine some basic properties of the sequence 
\begin{math} (c_n)_{n \in \naturals} \end{math}. 
We start with the following remark.

\begin{rem} The automaton \(A_3\) is synchronizing and the shortest synchronizing word is \(12\).
\end{rem}

This statement can be easily verified by hand -- we have 
\begin{math} \delta(q, 12) = [c_{9n+7}] \end{math} 
for all \(q \in Q\) and there is no other synchronizing word of length \(2\). Furthermore, we have 
\begin{math} \delta([c_{9n+7}], a) = [c_{9n+7}] \end{math}
for all \(a \in \Sigma_3\) and \(\tau([c_{9n+7}]) = 0\). We therefore have the following result.

\begin{samepage}
\begin{cor} \label{c_eq_zero} We have the following properties:
\begin{enumerate}[(a)]
\parskip-4pt
\item We have \(c_{9n+7} = 0\) for all \(n \in \naturals\).
\item \label{c_eq_part2} If the base-\(3\) representation of \(n\) contains \(21\), then \(c_n = 0\).
\end{enumerate}
\end{cor}
\end{samepage}

\begin{rem} The converse of \eqref{c_eq_part2} in the above corollary does not hold. For instance, we have \(c_{11} = 0\) and \((11)_3 = 102\). \end{rem}

Corollary \hyperref[c_eq_part2]{\ref*{c_eq_zero}(\ref*{c_eq_part2})} gives us the full information about the frequency of \(0\)'s in the sequence
\begin{math} (c_n)_{n \in \naturals} \end{math}
and the length of the strings of consecutive \(0\)'s. We introduce these properties in the following corollaries.

\begin{cor} \label{c_zero_seq} The sequence 
\begin{math} (c_n)_{n \in \naturals} \end{math}
contains arbitrarily long sequences of consecutive \(0\)'s.
\end{cor}

\begin{proof}
Let \(k \in \naturals\) and consider \(n = 7 \cdot 3^k\). The base-\(3\) expansion of \(n\) is
\begin{displaymath}
(n)_3 = 21\underbrace{00\dots 0}_k.
\end{displaymath}
Hence all the numbers between \(n\) and \(n + 3^k - 1\) have \(21\) as their leading digits. We thus have 
\begin{math} c_n = c_{n+1} = \dots = c_{n + 3^k - 1} = 0 \end{math}
by Corollary \hyperref[c_eq_part2]{\ref*{c_eq_zero}(\ref*{c_eq_part2})}.
\end{proof}

\begin{cor} \label{c_zero_freq} The frequency of \(0\)'s in the sequence 
\begin{math} (c_n)_{n \in \naturals} \end{math}
is equal to \(1\).
\end{cor}

\begin{proof}
Let \(n \in \naturals\) be nonzero. Then there exists \(m \in \naturals\) such that \begin{math} 9^{m-1} \leq n < 9^m \end{math}. We thus have that
\begin{displaymath}
0 \leq \frac{|\{k < n : c_k \neq 0\}|}{n} \leq \frac{|\{k < 9^m : c_k \neq 0\}|}{9^{m-1}}.
\end{displaymath}
The base-\(3\) representation of \(k < 9^m\) can be written as a word of length \(2m\), which can be divided into \(m\) pairs of letters. From Corollary \hyperref[c_eq_part2]{\ref*{c_eq_zero}(\ref*{c_eq_part2})}, in order for \(c_k\) to be nonzero, we need all these pairs to be different from \(21\). We have \(8\) possibilities for each pair. Hence we have at most \(8^m\) different values of \(k < 9^m\) such that \(c_k \neq 0\). We therefore have the equality
\begin{displaymath}
\frac{|\{k < 9^m : c_k \neq 0\}|}{9^{m-1}} \leq \frac{8^m}{9^{m-1}},
\end{displaymath}
which completes the proof, because 
\begin{math} \lim\limits_{m\to\infty}\frac{8^m}{9^{m-1}} = 0 \end{math}.
\end{proof}

In the next result, we are going to discuss the maximal number of consecutive \(1\)'s and \(2\)'s in the sequence
\begin{math} (c_n)_{n \in \naturals} \end{math}.
Unfortunately, it is not as easy as the case of consecutive \(0\)'s and it requires some advanced computations. First, we use \texttt{Walnut} to prove the following lemma.

\begin{samepage}
\begin{lem} \label{walnut1} The sequence
\begin{math} (c_n)_{n \in \naturals} \end{math}
has the following properties:
\begin{enumerate}[(a)]
\parskip-4pt
\item \label{w1_part1} For any \(n \geq 1\) there exists \(k \in \{-1, 0, \dots, 6\}\) such that \(c_{9n+k} = 0\).
\item \label{w1_part2} For any \(n \geq 1\) and \(k \in \{-1, 0, 1, 2\}\), if 
\begin{math} c_{9n+k} = c_{9n+k+1} = \hdots = c_{9n+k+4} \end{math},
then \(c_{9n+k} = 0\).
\end{enumerate}
\end{lem}
\end{samepage}

We are now ready to prove the following theorem.

\begin{thm} \label{c_consec_12} The maximal number of consecutive \(1\)'s and the maximal number of consecutive \(2\)'s in the sequence
\begin{math} (c_n)_{n \in \naturals} \end{math}
is equal to \(4\). Moreover, the sets
\begin{align*}
\mathcal{C}_1 &= \{n \in \naturals : c_n = c_{n+1} = c_{n+2} = c_{n+3} = 1\}, \\
\mathcal{C}_2 &= \{n \in \naturals : c_n = c_{n+1} = c_{n+2} = c_{n+3} = 2\}
\end{align*} 
are infinite.
\end{thm}

\begin{proof}
From Corollary \hyperref[c_eq_part2]{\ref*{c_eq_zero}(\ref*{c_eq_part2})} and Lemma \hyperref[w1_part2]{\ref*{walnut1}(\ref*{w1_part2})} we get that we cannot have \(5\) consecutive \(1\)'s or \(5\) consecutive \(2\)'s in the sequence
\begin{math} (c_n)_{n \in \naturals} \end{math}.
Hence, in order to complete the proof, we need to prove that the sets \(\mathcal{C}_1\) and \(\mathcal{C}_2\) are infinite.

We start with the set \(\mathcal{C}_1\). We consider a sequence
\begin{math} (g_n)_{n \in \naturals} \end{math}
given by
\begin{displaymath}
g_n = 4 \cdot 3^{n+2} + 1, \quad n \in \naturals.
\end{displaymath}
We have the equality
\begin{displaymath}
(g_n)_3 = 11\underbrace{00\dots 0}_{n+1}1.
\end{displaymath}
Hence, using the automaton \(A_3\), we obtain the formul\ae
\begin{align*}
\delta([c_n], (g_n)_3) = \delta([c_n], 11\underbrace{00\dots 0}_{n+1}1) &= \left\{\begin{array}{ll}
{[c_{9n+3}]} & \text{if } n \equiv 0, 1 \!\!\pmod 3, \\ 
{[c_{3n+1}]} & \text{if } n \equiv 2 \!\!\pmod 3,
\end{array}\right. \\
\delta([c_n], (g_n + 1)_3) = \delta([c_n], 11\underbrace{00\dots 0}_{n+1}2) &= \left\{\begin{array}{ll}
{[c_{9n+2}]} & \text{if } n \equiv 0 \!\!\pmod 3, \\
{[c_{3n+1}]} & \text{if } n \equiv 1, 2 \!\!\pmod 3, 
\end{array}\right. \\
\delta([c_n], (g_n + 2)_3) = \delta([c_n], 11\underbrace{00\dots 0}_{n}10) &= \left\{\begin{array}{ll}
{[c_{9n+3}]} & \text{if } n \equiv 0 \!\!\pmod 3, \\
{[c_{9n+2}]} & \text{if } n \equiv 1, 2 \!\!\pmod 3,
\end{array}\right. \\
\delta([c_n], (g_n + 3)_3) = \delta([c_n], 11\underbrace{00\dots 0}_{n}11) &= \left\{\begin{array}{ll}
{[c_{3n+1}]} & \text{if } n \equiv 0 \!\!\pmod 3, \\
{[c_{9n+3}]} & \text{if } n \equiv 1, 2 \!\!\pmod 3.
\end{array}\right.
\end{align*}
We therefore have 
\begin{math} \tau(\delta([c_n], (g_n + i)_3) = 1 \end{math}
for all \(n \in \naturals\) and \(i \in \{0, 1, 2, 3\}\), which means that \(g_n \in \mathcal{C}_1\) for all \(n \in \naturals\) and therefore \(\mathcal{C}_1\) is infinite. To prove that the set \(\mathcal{C}_2\) is also infinite, we use the same method as above, but instead of the sequence
\begin{math} (g_n)_{n \in \naturals} \end{math},
we consider a sequence 
\begin{math} (h_n)_{n \in \naturals} \end{math}
given by
\begin{displaymath}
h_n = 166 \cdot 3^{n+2} + 1
\end{displaymath}
for \(n \in \naturals\).
\end{proof}

Instead of considering the sequences of consecutive \(1\)'s and \(2\)'s in the sequence
\begin{math} (c_n)_{n \in \naturals} \end{math},
we can also consider the sequences of consecutive nonzero terms.

\begin{thm} The maximal number of consecutive nonzero terms in the sequence
\begin{math} (c_n)_{n \in \naturals} \end{math} 
is equal to \(7\). Moreover, the set
\begin{displaymath}
\mathcal{C}_3 = \{n \in \naturals : c_{n+i} \neq 0, \, i \in \{0, 1, \dots, 6\}\}
\end{displaymath}
is infinite.
\end{thm}

\begin{proof}
From Lemma \hyperref[w1_part1]{\ref*{walnut1}(\ref*{w1_part1})} we get that we cannot have more than \(7\) consecutive nonzero terms. Moreover, using the same method as before we can verify that
\begin{math} 4\cdot 3^{n+2} \in \mathcal{C}_3 \end{math}
for all \(n \in \naturals\).
\end{proof}

\begin{rem}
It is instructive to compare the sequence 
\begin{math} (c_n)_{n \in \naturals} \end{math}
with the original sequence 
\begin{math} (t_n^{(3)})_{n \in \naturals} \end{math}.
Recall that we have
\begin{displaymath}
t_0^{(3)} = 1, \quad t_{3n+i}^{(3)} = t_n^{(3)} + i, \quad i \in \{0, 1, 2\}.
\end{displaymath}
The sequence
\begin{math} (t_n^{(3)})_{n \in \naturals} \end{math}
has the following properties:
\begin{itemize}
\item The maximal numbers of consecutive \(0\)'s, consecutive \(1\)'s and consecutive \(2\)'s are all equal to \(2\).
\item The maximal number of consecutive nonzero terms is equal to \(4\).
\item The frequencies of \(0\)'s, \(1\)'s and \(2\)'s are all equal to \(\frac{1}{3}\).
\item The automaton generating this sequence is not synchronizing.
\end{itemize}
\end{rem}

We can also find analogous properties for the sequence
\begin{math} (t_n^{(p)})_{n \in \naturals} \end{math}
for all primes \(p\).

\subsection{\texorpdfstring{The case \(p = 5\)}{The case p=5}}
\label{section_thue5}

In this section we are going to study the case \(p = 5\) (sequence A053840 in \cite{OE}). From now on, to simplify the notation, we write
\begin{math} (d_n)_{n \in \naturals} \end{math}
instead of 
\begin{math} (c_n^{(5)})_{n \in \naturals} \end{math}.

We start with computing the automaton generating the sequence 
\begin{math} (d_n)_{n \in \naturals} \end{math}. 
From equation \eqref{Cp_equation}, we have that the formal power series 
\begin{math} C_5 = \sum_{n=0}^{\infty}d_nX^n \end{math}
satisfies the following algebraic relation:
\begin{equation} \label{C5_equation}
(1 - C_5)^6X^5 - (1 - C_5)^2X + C_5 = 0.
\end{equation}
From this equation, we can determine the automaton
\begin{math} A_5 = (Q, \Sigma_5, \delta, [d_n], \Sigma_5, \tau) \end{math}
generating the sequence
\begin{math} (d_n)_{n \in \naturals} \end{math}.
We again perform all the computations in Mathematica. The obtained automaton has \(2236\) states, which makes it much more complicated than the previously obtained automaton \(A_3\) for the case \(p = 3\). It also means that it would be inconvenient to represent the automaton \(A_5\) as a directed graph. We therefore represent it as a table, which describes all the connections between the states as well as the output values. Each state is labeled by a different number from \(1\) (the initial state) to \(2236\). Some rows of this table are presented below.

\begin{table}[ht!] 
\centering
\begin{tabular}{|c|c|c|c|c|c|c|}
\hline
\(q\) & \(\delta(q, 0)\) & \(\delta(q, 1)\) & \(\delta(q, 2)\) & \(\delta(q, 3)\) & \(\delta(q, 4)\) & \(\tau(q)\) \\
\hline
\(1\) & \(2\) & \(3\) & \(4\) & \(5\) & \(6\) & \(0\) \\
\hline
\(2\) & \(13\) & \(14\) & \(15\) & \(16\) & \(17\) & \(0\) \\
\hline
\(3\) & \(197\) & \(198\) & \(199\) & \(12\) & \(5\) & \(1\) \\
\hline
\(4\) & \(777\) & \(97\) & \(4\) & \(5\) & \(5\) & \(3\) \\
\hline
\(5\) & \(5\) & \(5\) & \(5\) & \(5\) & \(5\) & \(0\) \\
\hline
\(\vdots\) &&&&&& \\
\hline
\(2236\) & \(2161\) & \(137\) & \(266\) & \(120\) & \(4\) & \(2\) \\
\hline
\end{tabular}
\caption{The representation of the automaton \(A_5\).} \label{d_auto_table}
\end{table}

In the following, we use the above table to determine the properties of the sequence
\begin{math} (d_n)_{n \in \naturals} \end{math}.
We start with finding all the subautomata of the automaton \(A_5\). We coded a simple computer program to analyze all the connections. We found out that the automaton \(A_5\) has the following, \(4\)-level structure: 
\begin{enumerate}
\item The initial state \([d_n]\).
\item Two disjoint \(5\)-cycles: 
\begin{math} \{[d_{5^kn}] : 1 \leq k \leq 5\} \end{math}
and
\begin{math} \{[d_{5^kn+5^k-1}] : 1 \leq k \leq 5\} \end{math}.
\item The subset \(Q' \subset Q\) of \(2224\) states such that for any \(q_1, q_2 \in Q'\) there exists a word \(w \in \Sigma_5^\star\) such that \(\delta(q_1, w) = q_2\).
\item The terminal state \([d_{5n+3}]\) (for any \(a \in \Sigma_5\) we have 
\begin{math} \delta([d_{5n+3}], a) = [d_{5n+3}]) \end{math}
This state is represented by number \(5\) in Table \ref{d_auto_table}. 
\end{enumerate}
It is possible to go from any level of this structure to any lower level (considering the initial state as the highest level), but we cannot go from lower levels to higher ones. This structure is illustrated in Figure \ref{d_auto}. The output values are shown in Table \ref{d_out}.

\begin{figure}[ht!]
\begin{center}
\begin{tikzpicture}[->, shorten >= 1pt, node distance=2.5 cm, on grid, auto]
	\node[elliptic state, inner sep=2pt, line width = 1.5pt] (A) {\(d_n\)};
	\node[elliptic state, inner sep=2pt] (B1) [below=2 cm of A] {\(d_{5n}\)};
	\node[elliptic state, inner sep=2pt] (B2) [below=1.25cm of B1] {\(d_{25n}\)};
	\node[elliptic state, inner sep=2pt] (B3) [below=1.25cm of B2] {\(d_{125n}\)};
	\node[elliptic state, inner sep=2pt] (B4) [below=1.25cm of B3] {\(d_{625n}\)};
	\node[elliptic state, inner sep=2pt] (B5) [below=1.25cm of B4] {\(d_{3125n}\)};
	\node[elliptic state, inner sep=2pt] (C1) [above=4cm of B1] {\(d_{5n+4}\)};
	\node[elliptic state, inner sep=2pt] (C2) [above=1.25cm of C1] {\(d_{25n+24}\)};
	\node[elliptic state, inner sep=2pt] (C3) [above=1.25cm of C2] {\(d_{125n+124}\)};
	\node[elliptic state, inner sep=2pt] (C4) [above=1.25cm of C3] {\(d_{625n+624}\)};
	\node[elliptic state, inner sep=2pt] (C5) [above=1.25cm of C4] {\(d_{3125n+3124}\)};
	\node[elliptic state, inner sep=2pt] (D)
[left=5cm of A] {\(d_{5n+3}\)};
	\node[elliptic state, inner sep=12pt] (E)
[right=4cm of A] {\(Q'\)};
	\path[->]
	(A) edge node {0} (B1)
		edge node [swap] {4} (C1)
		edge node {1,2} (E)
		edge node [swap] {3} (D)
	(B1) edge node {0} (B2)
		 edge [bend right=30] node {} (E)
	(B2) edge node {0} (B3)
		 edge [bend right=30] node {} (E)
	(B3) edge node {0} (B4)
		 edge [bend right=30] node {} (E)
	(B4) edge node {0} (B5)
		 edge [bend right=30] node [swap] {1,2,3,4} (E)
	(B5) edge [bend left=50] node {0} (B1)
		 edge [bend right=70] node [swap] {1,2} (E)
		 edge [bend left=30] node {3,4} (D)
	(C1) edge node {4} (C2)
		 edge [bend left=30] node {} (E)
		 edge [bend right=60] node {} (D)
	(C2) edge node {4} (C3)
		 edge [bend left=30] node {} (E)
		 edge [bend right=60] node {} (D)
	(C3) edge node {4} (C4)
		 edge [bend left=30] node {} (E)
		 edge [bend right=60] node {} (D)
	(C4) edge node {4} (C5)
		 edge [bend left=30] node {} (E)
		 edge [bend right=60] node {} (D)
	(C5) edge [bend left=60] node {4} (C1)
		 edge [bend left=30] node {0,1,2} (E)
		 edge [bend right=60] node [swap] {3} (D)
	(D) edge [loop left] node {0,1,2,3,4} (D);
	\path[dashed,->]
	(E) edge [loop right] node {} (E)
	(E) edge [bend left=10] node {} (D);
\end{tikzpicture}
\vspace{-0.2 cm}
\caption{The structure of the automaton \(A_5\). \label{d_auto}}
\end{center}
\end{figure}
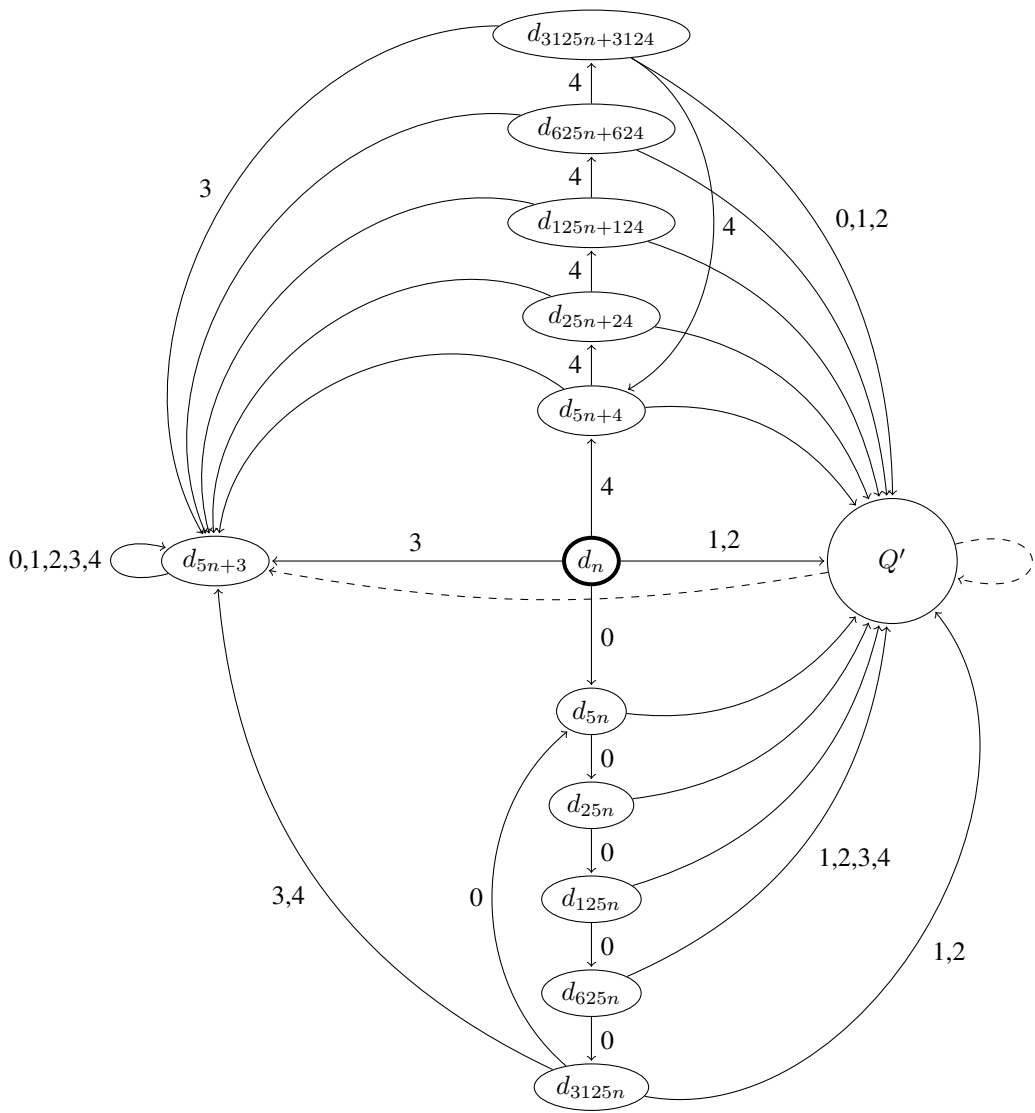

\begin{table}[ht!] 
\centering
\begin{tabular}{|c|c||c|c||c|c|}
\hline
\(q\) & \(\tau(q)\) & \(q\) & \(\tau(q)\) & \(q\) & \(\tau(q)\) \\
\hline
\([d_{n}]\) & \(0\) & \([d_{625n}]\) & \(0\) & \([d_{25n+24}]\) & \(2\) \\
\hline
\([d_{5n}]\) & \(0\) & \([d_{3125n}]\) & \(0\) & \([d_{125n+124}]\) & \(3\) \\
\hline
\([d_{25n}]\) & \(0\) & \([d_{5n+3}]\) & \(0\) & \([d_{625n+624}]\) & \(4\) \\
\hline
\([d_{125n}]\) & \(0\) & \([d_{5n+4}]\) & \(1\) & \([d_{3125n+3124}]\) & \(0\) \\
\hline
\end{tabular}
\caption{Some output values of the automaton \(A_5\). \label{d_out}}
\end{table}

The automaton \(A_5\) has the following property.

\begin{prop} The automaton \(A_5\) is synchronizing and the shortest synchronizing words are \(14\), \(24\), \(33\), \(34\) and \(43\). Moreover, for all \(q \in Q\) and \begin{math} w \in \{14, 24, 33, 34, 43\} \end{math}
we have 
\begin{math} \delta(q, w) = [d_{5n+3}] \end{math}.
\end{prop}

As an immediate consequence of the above proposition, we obtain the following result, which is analogous to results described in Corollary \ref{c_zero_seq} and Corollary \ref{c_zero_freq}.

\begin{thm} The sequence 
\begin{math} (d_n)_{n \in \naturals} \end{math}
contains arbitrarily long sequences of consecutive \(0\)'s and the frequency of \(0\)'s in this sequence is equal to \(1\). 
\end{thm}

As before, we can also consider the sequences of consecutive nonzero terms.

\begin{thm} The maximal numbers of consecutive \(1\)'s, \(2\)'s, \(3\)'s and \(4\)'s in the sequence 
\begin{math} (d_n)_{n \in \naturals} \end{math}
are all equal to \(4\). Moreover, the sets
\begin{displaymath}
\mathcal{D}_i = \{n \in \naturals : d_n = d_{n+1} = d_{n+2} = d_{n+3} = i\}, \quad i \in \{1,2,3,4\}
\end{displaymath}
are infinite.
\end{thm}

\begin{proof}
From the automaton \(A_5\) we easily get that \(d_{5n+3} = 0\) for all \(n \in \naturals\), hence we cannot have more than \(4\) consecutive nonzero terms. Therefore it suffices to prove that the sets \(\mathcal{D}_i\) are infinite. We define the sequences 
\begin{math} (k_n^{(i)})_{n \in \naturals} \end{math}
for \(i \in \{1,2,3,4\}\) in the following way:
\begin{align*}
(k_n^{(1)})_5 &= 100\underbrace{22\dots 2}_{n+1}44, \\
(k_n^{(2)})_5 &= 31\underbrace{22\dots 2}_{n+1}44, \\
(k_n^{(3)})_5 &= 210\underbrace{22\dots 2}_{n+1}44, \\
(k_n^{(4)})_5 &= 3021\underbrace{22\dots 2}_{n+1}44.
\end{align*}
Then, using the same method as in the proof of Theorem \ref{c_consec_12}, we can verify that 
\begin{math} k_n^{(i)} \in \mathcal{D}_i \end{math}
for all \(n \in \naturals\) and \(i \in \{1,2,3,4\}\) and that completes the proof.
\end{proof}

\begin{rem}
Since we have \(d_{5n+3} = 0\) for all \(n \in \naturals\), then the maximal number of consecutive nonzero terms in the sequence 
\begin{math} (d_n)_{n \in \naturals} \end{math}
is also equal to \(4\). This is different from the case \(p=3\) discussed earlier, when the maximal number of consecutive nonzero terms was almost twice as big as the maximal number of consecutive \(1\)'s and \(2\)'s.
\end{rem}

\subsection{Conjectures for the general case}

In this short section, we introduce some conjectures involving the properties of the sequence 
\begin{math} (c_n^{(p)})_{n \in \naturals} \end{math}
in the general case.

We used Mathematica to compute the automata generating the sequences 
\begin{math} (c_n^{(2)})_{n \in \naturals} \end{math}, 
\begin{math} (c_n^{(3)})_{n \in \naturals} \end{math} 
and 
\begin{math} (c_n^{(5)})_{n \in \naturals} \end{math}. 
The obtained automata have \(8\), \(28\) and \(2236\) states, respectively. Unfortunately, the case \(p=7\) is already out of the scope of computational abilities of Mathematica on our computer. However, it is still possible to give some conjectures about properties of the sequence 
\begin{math} (c_n^{(p)})_{n \in \naturals} \end{math}
in the general case -- we can use the recurrence relation from Proposition \ref{s_recurrence} to compute sufficiently large number of terms of the sequence for some \(p\) and then look for patterns.

Based on these computations, we formulate the following conjectures.

\begin{con} \label{conj_1}
We have \(c_{pn+i}^{(p)} = 0\) for all prime \(p > 3\), \(n \in \naturals\) and 
\begin{math} i \in \{\frac{p+1}{2}, \frac{p+3}{2}, \dots, \frac{p+(p-4)}{2}\} \end{math}.
\end{con}

\begin{con} \label{conj_2}
For any prime number \(p\), the frequency of \(0\)'s in the sequence 
\begin{math} (c_n^{(p)})_{n \in \naturals} \end{math}
is equal to \(1\) and we have arbitrarily long strings of consecutive \(0\)'s. Moreover, the maximal number of consecutive \(1\)'s, \(2\)'s, \(3\)'s, \dots, \((p-1)\)'s is finite and for \(p>3\) it is not greater than \(\frac{p+3}{2}\).
\end{con}

Note that if Conjecture \ref{conj_1} is true, then the second part of Conjecture \ref{conj_2} (the upper bound for \(p>3\)) is also true.

In the previous subsection, we discussed the structure of the automaton \(A_5\) generating the sequence 
\begin{math} (c_n^{(5)})_{n \in \naturals} \end{math}. 
Note that the automata \(A_2\) and \(A_3\) shown in Figure \ref{c2_auto} and Figure \ref{c_auto} have the similar structure. Moreover, all these automata are synchronizing. This observation allows us to formulate the final conjecture.

\begin{con}
For any prime number \(p\), the automaton 
\begin{math} A_p = (Q, \Sigma_p, \delta, [c_n^{(p)}], \Sigma_p, \tau) \end{math}
generating the sequence 
\begin{math} (c_n^{(p)})_{n \in \naturals} \end{math}
is synchronizing and it has the following structure:
\begin{enumerate}
\item The initial state \([c_n^{(p)}]\).
\item Two disjoint \(p\)-cycles: 
\begin{math} \{[c_{p^kn}^{(p)}] : 1 \leq k \leq p\} \end{math}
and
\begin{math} \{[c_{p^kn+p^k-1}^{(p)}] : 1 \leq k \leq p\} \end{math}.
\item The subset \(Q' \subset Q\) such that for any \(q_1, q_2 \in Q'\) there exists a word \(w \in \Sigma_p^\star\) such that \(\delta(q_1, w) = q_2\).
\item The terminal state \(\overline{q}\) (for any \(a \in \Sigma_p\) we have
\begin{math} \delta(\overline{q}, a) = \overline{q}) \end{math}.
\end{enumerate}
\end{con}

\section{The Rudin--Shapiro sequence} \label{section_rudin}
\subsection{Basic results}

The Rudin--Shapiro sequence is another common example of a \(2\)-automatic sequence. It takes only the values \(1\) and \(-1\) and its \(n\)-th term is defined as \(1\) if the number of (possibly overlapping) occurrences of \(11\) in the binary expansion of \(n\) is even, and \(-1\) otherwise.

In this section, we are going to use a slightly different version of the Rudin--Shapiro sequence, with all the \(-1\) terms changed into \(0\). We denote the new sequence by
\begin{math} (r_n)_{n \in \naturals} \end{math}. 
We regard its values \(r_n\) as elements in \(\mathbb{F}_2\). It is clear that the sequence 
\begin{math} (r_n)_{n \in \naturals} \end{math}
satisfies the following recurrence relations:
\begin{equation} \label{r_recurrence}
r_0 = 1, \quad r_{2n} = r_{4n+1} = r_n, \quad r_{4n+3} = 1 + r_{2n+1},
\end{equation}
for all \(n \in \naturals\). From these equations, it is easy to determine that the sequence 
\begin{math} (r_n)_{n \in \naturals} \end{math}
is generated by the following automaton:

\begin{figure}[ht!]
\begin{center}
\begin{tikzpicture}[->, shorten >= 1pt, node distance=2.5 cm, on grid, auto]
	\node[elliptic state, inner sep=3pt, line width = 1.5pt] (c_1) {\(r_n\)};
	\node[elliptic state, inner sep=2pt] (c_2) [right=of c_1] {\(r_{2n+1}\)};
	\node[elliptic state, inner sep=2pt] (c_3) [right=of c_2] {\(r_{4n+3}\)};
    \node[elliptic state, inner sep=2pt] (c_4) [right=of c_3] {\(r_{8n+3}\)};	
	\path[->]
	 (c_1) edge [loop above] node {0} (c_1)
	       edge [bend left] node {1} (c_2)
	 (c_2) edge [bend left] node {0} (c_1)
	       edge [bend left] node {1} (c_3)
	 (c_3) edge [bend left] node {0} (c_4)
	       edge [bend left] node {1} (c_2)
	 (c_4) edge [loop above] node {0} (c_4)
	       edge [bend left] node {1} (c_3);
\end{tikzpicture}
\vspace{-0.2 cm}
\caption{The automaton generating the sequence \((r_n)_{n \in \naturals}\). \label{r_auto}}
\end{center}
\end{figure}

Define the formal power series 
\begin{math} R = \sum_{n=0}^{\infty}r_nX^n \in \mathbb{F}_2[\![X]\!] \end{math}.
We can find an algebraic relation for the series \(R\) using equations \eqref{r_recurrence} (see \cite[Example 12.1.4]{AS} for details -- in this example, the letters \(0\) and \(1\) are interchanged but the recurrence relations are the same). The desired relation has the form
\begin{equation} \label{R_equation}
(1 + X)^5R^2 + (1 + X)^4R + X^3 = 0.
\end{equation}

Since \(r_0 = 1\), the composition inverse of the series \(R\) does not exist. In order for a composition inverse to exist, we need to modify the original sequence
\begin{math} (r_n)_{n \in \naturals} \end{math}.
We consider two different modifications: 
\begin{math} (r'_n)_{n \in \naturals} \end{math}
and
\begin{math} (r''_n)_{n \in \naturals} \end{math}.
Recall that these sequences are defined in the following way:
\begin{displaymath}
r'_n = \left\{ \begin{array}{ll}
0 & \text{ if } n = 0, \\ r_n & \text{ otherwise},
\end{array} \right.
\quad
r''_n = \left\{ \begin{array}{ll}
0 & \text{ if } n = 0, \\ r_{n-1} & \text{ otherwise}.
\end{array} \right.
\end{displaymath}

\subsection{\texorpdfstring{Formal inverse of the sequence \((r_n')_{n \in \naturals}\)}{Formal inverse of the sequence rn'}}

In the following part, we will focus on the sequence 
\begin{math} (r'_n)_{n \in \naturals} \end{math}
and its formal inverse. 

Let
\begin{math} R_1 = \sum_{n=0}^{\infty}r'_nX^n \in \mathbb{F}_2[\![X]\!] \end{math}.
It is clear that we have \(R_1 = R + 1\). Hence, from the equation \eqref{R_equation}, we obtain
\begin{equation} \label{R_equation1}
(R_1^2 + 1)X^5 + (R_1^2 + R_1)X^4 + X^3 + (R_1^2 + 1)X + R_1^2 + R_1 = 0.
\end{equation}

Denote the composition inverse of \(R_1\) by 
\begin{math} U = \sum_{n=0}^{\infty}u_nX^n \end{math}.
Left composing equation \eqref{R_equation1} with \(U\), we get the following equality:
\begin{equation} \label{U_equation}
(X^2 + 1)U^5 + (X^2 + X)U^4 + U^3 + (X^2 + 1)U + X^2 + X = 0.
\end{equation}
We can use equation \eqref{U_equation} to determine the automaton that generates the sequence
\begin{math} (u_n)_{n \in \naturals} \end{math}.
We again perform all the computations using Mathematica. The obtained automaton
\begin{math} A_U = (Q, \Sigma_2, \delta, [u_n], \Sigma_2, \tau) \end{math}
has \(23\) states and it is shown in Figure \ref{u_auto}. The output values of this automaton are shown in Table \ref{u_out}.

\begin{figure}
\begin{center}
\includegraphics[width=12 cm]{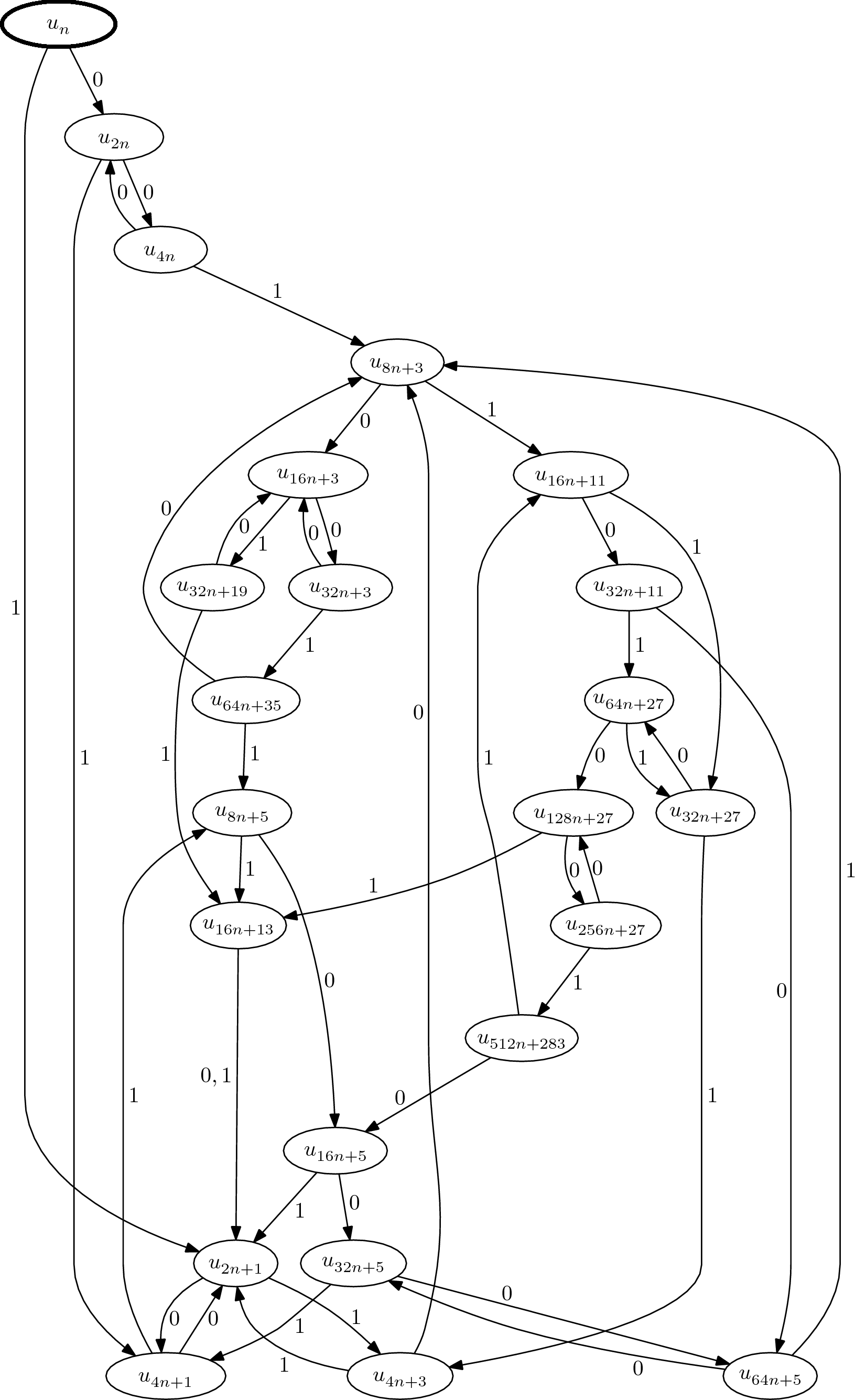}
\caption{The automaton \(A_U\) generating the sequence \((u_n)_{n \in \naturals}\). \label{u_auto}}
\end{center}
\end{figure}

\begin{table}[ht!] 
\centering
\begin{tabular}{|c|c||c|c||c|c||c|c|}
\hline
\(q\) & \(\tau(q)\) & \(q\) & \(\tau(q)\) & \(q\) & \(\tau(q)\) & \(q\) & \(\tau(q)\) \\
\hline
\([u_n]\) & \(0\) & \([u_{8n+3}]\) & \(0\) & \([u_{32n+3}]\) & \(0\) & \([u_{64n+27}]\) & \(0\) \\
\hline
\([u_{2n}]\) & \(0\) & \([u_{8n+5}]\) & \(1\) & \([u_{32n+5}]\) & \(1\) & \([u_{64n+35}]\) & \(0\) \\
\hline
\([u_{2n+1}]\) & \(0\) & \([u_{16n+3}]\) & \(1\) & \([u_{32n+11}]\) & \(1\) & \([u_{128n+27}]\) & \(0\) \\
\hline
\([u_{4n}]\) & \(0\) & \([u_{16n+5}]\) & \(1\) & \([u_{32n+19}]\) & \(0\) & \([u_{256n+27}]\) & \(0\) \\
\hline
\([u_{4n+1}]\) & \(1\) & \([u_{16n+11}]\) & \(1\) & \([u_{32n+27}]\) & \(0\) & \([u_{512n+283}]\) & \(1\) \\
\hline
\([u_{4n+3}]\) & \(0\) & \([u_{16n+13}]\) & \(1\) & \([u_{64n+5}]\) & \(0\) & & \\
\hline
\end{tabular}
\caption{The output values of the automaton \(A_U\). \label{u_out}}
\end{table}

Note that this automaton is not synchronizing (in contrast to the automata obtained in previous section). To prove this fact, we consider the following subset of states:
\begin{align*}
Q' = \{&[u_n], [u_{4n}], [u_{4n+1}], [u_{4n+3}], [u_{16n+3}], [u_{16n+5}], [u_{16n+11}], \\
&[u_{16n+13}], [u_{64n+5}], [u_{64n+27}], [u_{64n+35}], [u_{256n+27}]\}.
\end{align*}
Let \(Q'' = Q \setminus Q'\). From the automaton \(A_U\) we get that
\begin{align*}
\{\delta(q, 0), \delta(q, 1)\} \subset Q'' \quad &\text{for all } q \in Q', \\
\{\delta(q, 0), \delta(q, 1)\} \subset Q' \quad &\text{for all } q \in Q'',
\end{align*}
hence there is no such a word \(w \in \Sigma_2^{\star}\) such that \(\delta(q, w)\) is the same for all \(q \in Q\). 

\begin{rem}
Based on the automaton \(A_U\), one can easily construct another automaton generating the sequence 
\begin{math} (u_n)_{n \in \naturals} \end{math}
with input in base \(4\). This automaton has \(12\) states, corresponding to the states from the set \(Q'\) defined above. It can be shown that this automaton is synchronizing and the shortest synchronizing word is \(33\).
\end{rem}

We will now focus on the properties of the sequence
\begin{math} (u_n)_{n \in \naturals} \end{math},
based on the automaton \(A_U\). We start with the following result.

\begin{thm} \label{u_consec}
The sequence
\begin{math} (u_n)_{n \in \naturals} \end{math}
contains arbitrarily long sequences of consecutive \(0\)'s and arbitrarily long sequences of consecutive \(1\)'s.
\end{thm}

\begin{proof} We use the automaton \(A_U\). It can be verified that for all \(q \in Q\) we have
\begin{equation} \label{states_eq}
\delta(q, 1111001) \in \{[u_{4n+3}], [u_{32n+19}]\}.
\end{equation}

Let \(k \in \naturals\) and consider \(n = 79 \cdot 2^k\). Then
\begin{displaymath}
(n)_2 = 1001111\underbrace{00\dots 0}_k.
\end{displaymath}
From \eqref{states_eq} and the output values in Table \ref{u_out}, we have
\begin{displaymath}
u_{n+m} = \tau(\delta([u_n], (n+m)_2^R)) \in \{\tau([u_{4n+3}]), \tau([u_{32n+19}])\} = \{0\}
\end{displaymath}
for
\begin{math} m \in \{0, 1, \dots, 2^k-1\} \end{math}.
Hence 
\begin{math} u_{79\cdot 2^k}, u_{79\cdot 2^k+1}, \dots, u_{80\cdot 2^k-1} \end{math}
is a string of \(2^k\) consecutive \(0\)'s for any \(k \in \naturals\). We can find a string of \(2^k\) consecutive \(1\)'s using the same method. In this case we have
\begin{displaymath}
\delta(q, 111101) \in \{[u_{8n+5}], [u_{16n+11}]\}
\end{displaymath}
for all \(q \in Q\) and 
\begin{math} \tau([u_{8n+5}]) = \tau([u_{16n+11}]) = 1 \end{math}.
The sequence of \(2^k\) consecutive \(1\)'s is therefore given by \begin{math} u_{47\cdot 2^k}, \, u_{47\cdot 2^k+1}, \dots u_{48\cdot 2^k-1} \end{math}.
\end{proof}

Note that in the original Rudin--Shapiro sequence 
\begin{math} (r_n)_{n \in \naturals} \end{math}
the maximal number of consecutive \(0\)'s and consecutive \(1\)'s is equal to \(4\) and the strings of that length appear in the sequence infinitely many times. Indeed, from equations \eqref{r_recurrence} we get that 
\begin{math} \{r_{4n+2}, r_{4n+3}\} = \{0, 1\} \end{math}
for all \(n \in \naturals\) and therefore we cannot have more than \(4\) consecutive \(0\)'s or \(1\)'s. On the other hand, let
\begin{displaymath}
p_n = 3 \cdot 4^{n+2} - 1, \quad q_n = 2 \cdot 4^{n+1} - 1,
\end{displaymath}
for \(n \in \naturals\). Using the automaton shown in Figure \ref{r_auto}, one can prove that
\begin{align*}
r_{p_n} &= r_{p_n + 1} = r_{p_n + 2} = r_{p_n + 3} = 0, \\
r_{q_n} &= r_{q_n + 1} = r_{q_n + 2} = r_{q_n + 3} = 1,
\end{align*}
for all \(n \in \naturals\).

As a consequence from Theorem \ref{u_consec}, we get the following result.

\begin{cor} \label{u_frequency} In the sequence 
\begin{math} (u_n)_{n \in \naturals} \end{math},
the frequencies of \(0\)'s and \(1\)'s do not exist.
\end{cor}

\begin{proof} We prove that the frequency of \(1\)'s does not exist. Let \(k \in \naturals\) and denote the number of \(1\)'s among the first \(k\) terms of the sequence
\begin{math} (u_n)_{n \in \naturals} \end{math}
by \(N_k\). From the proof of Theorem \ref{u_consec} we have \begin{math} N_{79\cdot2^k} = N_{80\cdot2^k} \end{math}. Therefore, we get
\begin{displaymath}
\frac{N_{79 \cdot 2^k}}{79 \cdot 2^k} = \frac{80}{79} \cdot \frac{N_{80 \cdot 2^k}}{80 \cdot 2^k}.
\end{displaymath}
Hence, if the limit
\begin{displaymath}
\lim_{k \to \infty}\frac{N_k}{k}
\end{displaymath}
exists, it has to be equal to \(0\). Therefore, if the frequency of \(1\)'s exists, it is equal to \(0\). However, we can use similar reasoning for the frequency of \(0\)'s and deduce that it is also equal to \(0\), which is a contradiction.
\end{proof}

\begin{rem}
This is in the strong contrast with the original Rudin--Shapiro sequence. One can prove that in the sequence
\begin{math} (r_n)_{n \in \naturals} \end{math}
the frequencies of \(0\)'s and \(1\)'s exist and they are both equal to \(\frac{1}{2}\).
\end{rem}

\subsection{\texorpdfstring{Formal inverse of the sequence \((r_n'')_{n \in \naturals}\)}{Formal inverse of the sequence rn''}}

In the next part of this section, we will introduce the properties and the recurrence relations for the formal inverse of the sequence 
\begin{math} (r''_n)_{n \in \naturals} \end{math}. 

Let
\begin{math} R_2 = \sum_{n=0}^{\infty}r''_nX^n \in \mathbb{F}_2[\![X]\!] \end{math}.
Then it is clear that \(R_2 = XR\). Denote the composition inverse of the series \(R_2\) by 
\begin{math} V = \sum_{n=0}^{\infty}v_nX^n \end{math}.
From equation \eqref{R_equation} we get
\begin{equation} \label{R_equation2}
(R_2^2 + R_2 + 1)X^5 + R_2^2X^4 + (R_2^2 + R_2)X + R_2^2 = 0.
\end{equation}
Hence, after left composing equation \eqref{R_equation2} with \(V\), we obtain
\begin{equation} \label{V_equation}
(X^2 + X + 1)V^5 + X^2V^4 + (X^2 + X)V + X^2 = 0.
\end{equation}

Equation \eqref{V_equation} can be used to find the automaton
\begin{math} A_V = (Q, \Sigma_2, \delta, [v_n], \Sigma_2, \tau) \end{math}
that generates the sequence 
\begin{math} (v_n)_{n \in \naturals} \end{math}.
As before, all the computations were performed in Mathematica. The obtained automaton has \(33\) states. It is shown in Figure \ref{v_auto} with output values shown in Table \ref{v_out}.

\begin{figure}
\begin{center}
\includegraphics[width=12 cm]{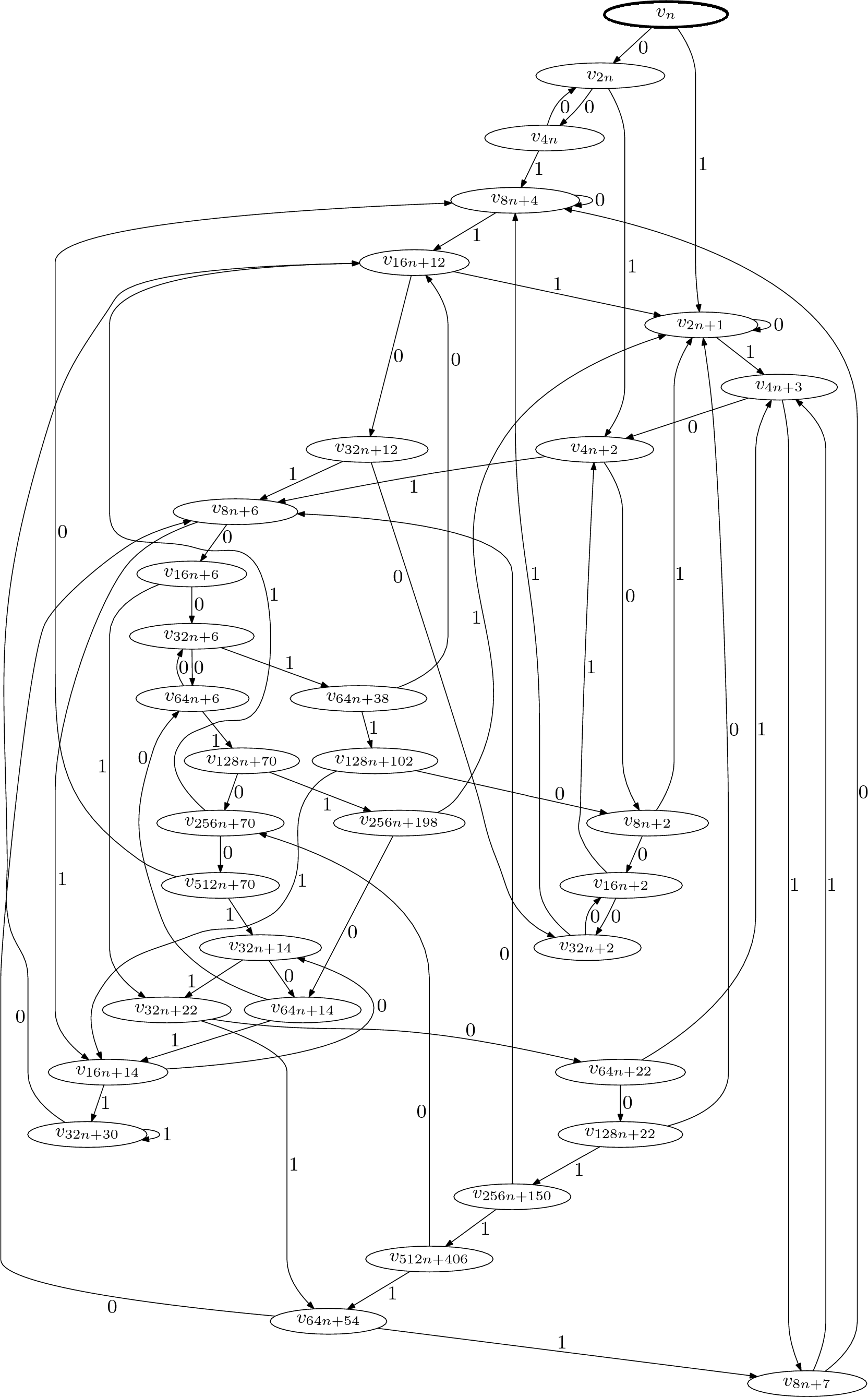}
\caption{The automaton \(A_V\) generating the sequence \((v_n)_{n \in \naturals}\). \label{v_auto}}
\end{center}
\end{figure}

\begin{table}[ht!] 
\centering
\begin{tabular}{|c|c||c|c||c|c|}
\hline
\(q\) & \(\tau(q)\) & \(q\) & \(\tau(q)\) & \(q\) & \(\tau(q)\) \\
\hline
\([v_n]\) & \(0\) & \([v_{16n+6}]\) & \(0\) & \([v_{64n+22}]\) & \(1\) \\
\hline
\([v_{2n}]\) & \(0\) & \([v_{16n+12}]\) & \(1\) & \([v_{64n+38}]\) & \(1\) \\
\hline
\([v_{2n+1}]\) & \(1\) & \([v_{16n+14}]\) & \(0\) & \([v_{64n+54}]\) & \(0\) \\
\hline
\([v_{4n}]\) & \(0\) & \([v_{32n+2}]\) & \(1\) & \([v_{128n+22}]\) & \(1\) \\
\hline
\([v_{4n+2}]\) & \(1\) & \([v_{32n+6}]\) & \(0\) & \([v_{128n+70}]\) & \(0\) \\
\hline
\([v_{4n+3}]\) & \(1\) & \([v_{32n+12}]\) & \(1\) & \([v_{128n+102}]\) & \(1\) \\
\hline
\([v_{8n+2}]\) & \(1\) & \([v_{32n+14}]\) & \(0\) & \([v_{256n+70}]\) & \(0\) \\
\hline
\([v_{8n+4}]\) & \(0\) & \([v_{32n+22}]\) & \(1\) & \([v_{256n+150}]\) & \(0\) \\
\hline
\([v_{8n+6}]\) & \(0\) & \([v_{32n+30}]\) & \(1\) & \([v_{256n+198}]\) & \(0\) \\
\hline
\([v_{8n+7}]\) & \(0\) & \([v_{64n+6}]\) & \(0\) & \([v_{512n+70}]\) & \(0\) \\
\hline
\([v_{16n+2}]\) & \(1\) & \([v_{64n+14}]\) & \(0\) & \([v_{512n+406}]\) & \(0\) \\
\hline
\end{tabular}
\caption{The output values of the automaton \(A_V\). \label{v_out}}
\end{table}

In the last part of this section, we will prove some properties of the sequence 
\begin{math} (v_n)_{n \in \naturals} \end{math}
analogous to the properties of the sequence
\begin{math} (u_n)_{n \in \naturals} \end{math}
presented in Theorem \ref{u_consec} and Corollary \ref{u_frequency}.

\begin{thm} \label{v_consec} The sequence
\begin{math} (v_n)_{n \in \naturals} \end{math}
contains arbitrarily long sequences of consecutive \(0\)'s.
\end{thm}

\begin{proof} The proof is similar to the proof of Theorem \ref{u_consec}. In this case, for all \(q \in Q\) we have
\begin{displaymath}
\delta(q, 101011) \in \{[v_{8n+7}], [v_{16n+14}], [v_{64n+54}]\}
\end{displaymath}
and
\begin{math} \tau([v_{8n+7}]) = \tau([v_{16n+14}]) = \tau([v_{64n+54}]) = 0 \end{math}.
Hence, for a given \(k \in \naturals\), we have a string of \(2^k\) consecutive \(0\)'s given by
\begin{math} v_{53\cdot2^k},\, v_{53\cdot2^k + 1}, \dots, v_{54\cdot2^k - 1} \end{math}.
\end{proof}

We can try to use the same method again to prove that in the sequence
\begin{math} (v_n)_{n \in \naturals} \end{math}
there are arbitrarily long strings of consecutive \(1\)'s. In other words, we need to find a word \(w \in \Sigma_2^\star\) such that
\begin{math} \tau(\delta(q, w)) = 1 \end{math}
for all \(q \in Q\). However, it turns out that in this case such a word does not exist and in fact the length of the string of consecutive \(1\)'s in the sequence
\begin{math} (v_n)_{n \in \naturals} \end{math}
is bounded. 

In order to prove this result, we start with the following lemma, which can be easily proved using \texttt{Walnut}.

\begin{lem} \label{walnut2} For all \(n \in \naturals\), at least one of the terms 
\begin{math} v_{4n+2}, \,v_{4n+3},\, v_{4n+4},\, v_{4n+5} \end{math}
is equal to \(0\).
\end{lem}

As a consequence, we get the following.
\begin{thm} The maximal number of consecutive \(1\)'s in the sequence
\begin{math} (v_n)_{n \in \naturals} \end{math}
is equal to \(6\). Moreover, the set
\begin{displaymath}
\mathcal{V} = \{n \in \naturals : v_{n+i} = 1, \, i = 0, 1, \dots, 5\}
\end{displaymath}
is infinite.
\end{thm}

\begin{proof} By Lemma \ref{walnut2} we know that we cannot have more than \(6\) consecutive \(1\)'s in the sequence
\begin{math} (v_n)_{n \in \naturals} \end{math}.
Moreover, using the same method as in the proof of Theorem \ref{c_consec_12}, it can be checked that
\begin{math} 3\cdot 4^{n+3}+11 \in \mathcal{V} \end{math}
for all \(n \in \naturals\).
\end{proof}

\begin{rem} As an immediate consequence of the above result, we get that the automaton \(A_V\) is not synchronizing. Indeed, if it were synchronizing, we could use the same argument as in the proof of Theorem \ref{v_consec} to show that the sequence
\begin{math} (v_n)_{n \in \naturals} \end{math}
contains arbitrarily long series of consecutive \(1\)'s.
\end{rem}

In the previous subsection, we proved that the frequencies of \(0\)'s and \(1\)'s in the sequence
\begin{math} (u_n)_{n \in \naturals} \end{math}
do not exist. It turns out that the same is also true for the sequence
\begin{math} (v_n)_{n \in \naturals} \end{math}.
We are going to finish this section with the proof of that fact. We start with the following lemma.

\begin{lem} If the frequency of \(1\)'s in the sequence
\begin{math} (v_n)_{n \in \naturals} \end{math}
exists, then it is equal to \(0\). 
\end{lem}

The proof of this lemma is analogous to the proof of Corollary \ref{u_frequency} (here we use Theorem \ref{v_consec} instead of Theorem \ref{u_consec}). On the other hand, we have the following result.

\begin{lem} If the frequency of \(1\)'s in the sequence
\begin{math} (v_n)_{n \in \naturals} \end{math}
exists, then it is not smaller than \(\frac{1}{11}\).
\end{lem}

\begin{proof} It can be verified that for all \(q \in Q\) we have
\begin{displaymath}
1 \in \{\tau(\delta(q, 0001)), \tau(\delta(q, 0101))\}.
\end{displaymath}
Therefore, for a given \(k \in \naturals\), we know that among the terms
\begin{displaymath}
v_{8 \cdot 2^k}, v_{8 \cdot 2^k + 1}, \dots, v_{9 \cdot 2^k - 1}, v_{10 \cdot 2^k}, v_{10 \cdot 2^{k-1}}, \dots, v_{11 \cdot 2^{k-1}}
\end{displaymath}
we have at least \(2^k\) terms that are equal to \(1\). Hence, we have at least \(2^k\) terms equal to \(1\) in the first \(11 \cdot 2^k\) terms of the sequence
\begin{math} (v_n)_{n \in \naturals} \end{math},
which completes the proof.
\end{proof} 

At the end, we arrive at the following corollary.

\begin{cor} In the sequence
\begin{math} (v_n)_{n \in \naturals} \end{math},
the frequencies of \(0\)'s and \(1\)'s do not exist.
\end{cor}

\acknowledgements
I would like to express my sincere gratitude to my Master's thesis supervisor, dr Jakub Byszewski, for numerous suggestions, comments and the continuous support. I also thank my friend Bart\l{}omiej Puget for providing me tools to create the graphs of the automata in Sections \ref{section_thue} and \ref{section_rudin}.

\nocite{*}
\bibliographystyle{plain}
\bibliography{sequences}

\begin{thebibliography}{1}

\bibitem{THUE}
Jean-Paul Allouche and Jeffrey Shallit.
\newblock The ubiquitous prouhet-thue-morse sequence.
\newblock In {\em Sequences and their applications}, pages 1--16. Springer,
  1999.

\bibitem{AS}
Jean-Paul Allouche, Jeffrey Shallit, et~al.
\newblock {\em Automatic sequences: theory, applications, generalizations}.
\newblock Cambridge university press, 2003.

\bibitem{PTM}
Maciej Gawron and Maciej Ulas.
\newblock On formal inverse of the prouhet--thue--morse sequence.
\newblock {\em Discrete Mathematics}, 339(5):1459--1470, 2016.

\bibitem{LANG}
Serge Lang.
\newblock {\em Complex analysis}, volume 103.
\newblock Springer Science \& Business Media, 2013.

\bibitem{MT}
{\L}ukasz Merta.
\newblock Composition inverses of certain automatic power series.
\newblock Master's thesis, Jagiellonian University in Krak\'{o}w, Poland, 2017.

\bibitem{BS}
{\L}ukasz Merta.
\newblock Composition inverses of the variations of the baum--sweet sequence.
\newblock {\em Theoretical Computer Science}, 784:81--98, 2019.

\bibitem{WAL}
Hamoon Mousavi.
\newblock Automatic theorem proving in walnut.
\newblock {\em arXiv preprint arXiv:1603.06017}, 2016.

\bibitem{PD}
Narad Rampersad and Manon Stipulanti.
\newblock The formal inverse of the period-doubling sequence.
\newblock {\em Journal of Integer Sequences}, 21(2):3, 2018.

\bibitem{OE}
NJA Sloane.
\newblock Oeis foundation inc.,‘the on-line encyclopedia of integer
  sequences’,(2017).

\end{thebibliography}
\end{document}